\documentclass[conference, 10pt, twocolumn]{ieeeconf}       % Use this line for a4
\bstctlcite{bstctl:etal, bstctl:nodash, bstctl:simpurl}
\def\BibTeX{{\rm B\kern-.05em{\sc i\kern-.025em b}\kern-.08em
    T\kern-.1667em\lower.7ex\hbox{E}\kern-.125emX}}
    
\usepackage[left=54pt, right=54pt,bottom=54pt, top=54pt]{geometry}
\usepackage{amsmath,mathrsfs,amsfonts,amssymb,graphicx,epsfig}
 \usepackage{amsthm}
\usepackage{subcaption}
\usepackage{color,multirow,rotating,tcolorbox}
\usepackage{algorithm,algpseudocode,algorithmicx}
\usepackage{cite,url,framed,bm,balance,nicematrix,physics}
\usepackage{stmaryrd,mathtools}

\newtheorem{theorem}{Theorem}

\newtheorem{lemma}{Lemma}
\newtheorem{proposition}{Proposition}
\newtheorem{remark}{Remark}

\newtheorem{mytheorem}{Theorem}[]      % Create new theorem called 'Theorem'.

\usepackage{tikz}
\usetikzlibrary{calc,trees,positioning,arrows,chains,shapes.geometric,decorations.pathreplacing,decorations.pathmorphing,shapes, matrix,shapes.symbols}

\usepackage[breaklinks]{hyperref}

\newcommand{\Jac}{{\mathrm{D}}}
\renewcommand{\det}{{\mathrm{det}}}

\usepackage[dvipsnames]{xcolor}
\usepackage{soul}
\sethlcolor{MidnightBlue!10}

\makeatletter
\newcommand{\pushright}[1]{\ifmeasuring@#1\else\omit\hfill$\displaystyle#1$\fi\ignorespaces}

\newcommand{\dotminus}{\mathbin{\text{\@dotminus}}}

\newcommand{\@dotminus}{%
  \ooalign{\hidewidth\raise1ex\hbox{.}\hidewidth\cr$\m@th-$\cr}%
}

\allowdisplaybreaks

\title{\LARGE\textbf{Linear Exponential Quadratic Gaussian Covariance Steering
}}

\author{Chiran B. Cherian, Yasemin Isik, Abhishek Halder% <-this % stops a space
\thanks{Chiran B. Cherian and Abhishek Halder are with the Department of Aerospace Engineering, Iowa State University, Ames, IA 50011, USA, {\tt\footnotesize{\{cbckbc,ahalder\}@iastate.edu}}. Yasemin Isik is with the Department of Mathematics, Iowa State University, Ames, IA 50011, USA, {\tt\footnotesize{yisik@iastate.edu}}.
}
\thanks{This research was partially supported by NSF awards 2111688, 2450377.}
}

\IEEEoverridecommandlockouts
\begin{document}
\bstctlcite{IEEEexample:BSTcontrol}
\maketitle
\thispagestyle{empty}
\pagestyle{empty}

%%%%%%%%%%%%%%%%%%%%%%%%%%%%%%%%%%%%%%%%%%%%%%%%%%%%%%%%%%%%%%%%%%%%%%%%%%%%%%%%
\begin{abstract}
We formulate and analyze the linear exponential quadratic Gaussian (LEQG) covariance steering problem in continuous time over a given deadline (finite time horizon). The solution for this problem can be seen as a risk-sensitive Schr\"{o}dinger bridge between Gaussian endpoints in the linear quadratic setting. Unlike the risk-neutral case, the LEQG covariance steering controller--still a linear state feedback--can no longer be written in closed form. We show that the optimal controller is parameterized by a symmetric matrix solving an algebraic equation that encodes the implicit dependence on the risk-sensitivity parameter. We explain how the structure of this optimal controller significantly generalizes the existing results for the risk-neutral case. Building on these results, for the matched noise and input channel case, we prove the existence-uniqueness of solution for the LEQG covariance steering problem in the neighborhood of the known risk-neutral optimal solution. We give an illustrative numerical example.
\end{abstract}

\section{Introduction}\label{sec:Intro}
For the standard integral quadratic cost
\begin{align}
C:=\!\int_{0}^{1}\!\frac{1}{2}\left( x_t^{\top}Q_t x_t + u_t^{\top}R_t u_t  \right)\differential t, \; Q_t \succeq 0, \; R_t \succ 0,
\label{defIQC}    
\end{align}
and fixed finite $\theta\in\mathbb{R}$, we consider the \emph{linear exponential quadratic Gaussian (LEQG) covariance steering problem}:
\begin{subequations}
\begin{align}
\inf_{\gamma\in \Gamma} \qquad&\frac{1}{\theta} \log \mathbb{E} \left[ \exp \left( \theta C \right) \right]\label{LEQGobjective}\\
\text{subject to}\quad &\differential x_t = \left(A_t x_t + B_t u_t\right)\differential t + F_{t}\,\differential w_{t},\label{LTVdynconstr2}\\
&x_0\sim \rho_0 := \mathcal{N}(0,\Sigma_0),
    \quad
    x_1\sim \rho_1:=\mathcal{N}(0,\Sigma_1),\label{EndpointConstr2}
\end{align}
\label{LEQGcovariancesteeringproblem}
\end{subequations}
in continuous time $t\in[0,1]$, where the state $x_t\in\mathbb{R}^{n}$, the control input $u_t\in\mathbb{R}^{m}$, and the standard Wiener process $w_t\in\mathbb{R}^{p}$. Problem \eqref{LEQGcovariancesteeringproblem} seeks to design a Markovian feedback controller $\gamma:[0,1]\times\mathbb{R}^{n}\mapsto\mathbb{R}^{m}$, i.e., $u_t = \gamma(t,x)$, such that the optimally controlled state is steered between the centered Gaussians $\rho_0,\rho_1$ with given covariances $\Sigma_0,\Sigma_1$ over the given time horizon $[0,1]$. 

The set of feasible control policies is $$\Gamma:=\{\gamma:[0,1]\times\mathbb{R}^{n}\mapsto\mathbb{R}^{m} \mid \mathbb{E}\int_{0}^{1}\|\gamma(t,x)\|_2^2\differential t < \infty\},$$
    where the expectation operator $\mathbb{E}\left[\cdot\right]$ is taken w.r.t. the controlled state under the chosen policy $\gamma$. For a given $\gamma\in\Gamma$, the $\mathbb{E}\left[\cdot\right]$ in \eqref{LEQGobjective} is w.r.t. the law of the controlled state. 
    
    There is no loss of generality in considering the endpoint Gaussians in \eqref{EndpointConstr2} to be centered because the solution for the nonzero-mean case is obtained by adding a time-varying drift to the optimal controller for problem \eqref{LEQGcovariancesteeringproblem}. This time-varying drift depends on the solution of problem \eqref{LEQGcovariancesteeringproblem} but not the other way around; see \cite[Remark 9]{chen2015optimalPartI} mutatis mutandis for $\theta\neq 0$. Throughout, we make the following assumptions.\\
\noindent\textbf{Assumptions.}
\begin{itemize}
    
    \item[A1.] The trajectory tuple $(A_t,B_t,F_t)$ is  bounded and continuous in $t\in[0,1]$, and the pair $(A_t,B_t)$ is uniformly controllable over the time horizon $[0,1]$.

    \item[A2.] The matricial trajectories $Q_t \succeq 0$ and $R_t\succ 0$ are bounded and continuous w.r.t. $t\in[0,1]$. 

    \item[A3.] The given covariance matrices $\Sigma_0,\Sigma_1\succ 0$.
\end{itemize}

Problem \eqref{LEQGcovariancesteeringproblem} can be seen as a risk-sensitive generalization of the \emph{risk-neutral linear quadratic Gaussian (LQG) covariance steering problem}--the latter differs from \eqref{LEQGcovariancesteeringproblem} only in that the objective then is replaced by the average cost $\mathbb{E}[C]$. The constant $\theta\in\mathbb{R}$ is called the \emph{risk-sensitivity parameter}. Notice that $\log \mathbb{E} \left[ \exp \left( \theta C \right) \right]$ is the logarithmic moment generation function (log-MGF) of the random variable $C$.

By L'Hôpital's rule, $\lim_{\theta \to 0} \frac{1}{\theta} \log \mathbb{E}[e^{\theta C}]  = \mathbb{E}[C]$, i.e., the limit $\theta \rightarrow 0$ corresponds to the \emph{risk-neutral} LQG objective. When $\theta > 0$, the LEQG objective becomes \emph{risk-averse}. When $\theta < 0$, the objective becomes \emph{risk-seeking}.

While the LQG problem minimizes the average cost $\mathbb{E}[C]$, the LEQG objective is such that the realizations of the (random) cost $C$ that are above average, are penalized more than the realizations of $C$ that are below average. So minimizing $\frac{1}{\theta} \log \mathbb{E}[e^{\theta C}]$ ensures that for $\theta>0$ large, a larger $C$ is more aggressively penalized. The motivation is that if the random variable $C$ has a large variance, then we prefer a more conservative controller that might be worse on average, but perform better in the worst case.

\noindent\textbf{Related Works.} In classical stochastic control, the solution of the LEQG problem \emph{with terminal cost} is well-known via dynamic programming \cite[Sec. VIII]{jacobson1973optimal}, via stochastic maximum principle \cite{whittle1990risk,lim2005new}, and via completion-of-square and Radon-Nikodym derivative \cite{duncan2013linear}. Notice, however, that problem \eqref{LEQGcovariancesteeringproblem} has no terminal cost and instead, has endpoint covariance constraints.

On the other hand, the \emph{LQG covariance steering} problem (i.e., with objective $\mathbb{E}[C]$) was first studied in \cite{chen2015optimalPartI,chen2015optimalPartII,chen2018optimal}. Specifically, the work in \cite{chen2015optimalPartI} derived the optimal controller for the case $F_t = B_t, Q_t \equiv 0$. The developments in \cite{chen2015optimalPartII} considered different input and noise coefficient matrices (i.e., $F_t\neq B_t$) and $Q_t \equiv 0$. The work in \cite{chen2018optimal} obtained the solution for $F_t = B_t$ and $Q_t \succeq 0$. The purpose of our work is to initiate a generalization of these results for the risk-sensitive objective $\frac{1}{\theta} \log \mathbb{E}[e^{\theta C}]$.

From a probabilistic viewpoint, problem \eqref{LEQGcovariancesteeringproblem} can be interpreted as a risk-sensitive variant of the Schr\"{o}dinger bridge problem in the linear quadratic setting \cite{chen2021stochastic,11239424}:
\begin{subequations}
\begin{align}
&\inf_{(\rho,\gamma)\in \mathcal{P}_{01}\times\Gamma} \quad\frac{1}{\theta} \log \mathbb{E} \left[ \exp \left( \theta C \right) \right]\label{RiskSensititiveLQSBobjective}\\
&\;\text{subject to}\quad\partial_t\rho + \nabla_{x}\cdot\left(\rho\left(A_t x + B_t\gamma\right)\right) = \frac{1}{2}\langle F_t F_t^{\top},\nabla_{x}^{2}\rho\rangle,\label{FPKlinear}
\end{align}
\label{RiskSensitiveSB}
\end{subequations}
where $\mathcal{P}_{01}$ is the collection of probability density function (PDF)-valued curves that are continuous in $t\in[0,1]$, and connect the given centered Gaussian endpoints $\rho_0,\rho_1$. In \eqref{FPKlinear}, the symbol $\langle\cdot,\cdot\rangle$ denotes the Frobenius inner product, and $\nabla_{x}\cdot, \nabla_{x}^2$ denote the Euclidean divergence and Hessian, respectively. In particular, \eqref{FPKlinear} is the Fokker-Planck or Kolmogorov's forward PDE describing the evolution of state PDF $\rho(t,x)$ under the controlled dynamics. In this letter, we will not develop this connection (i.e., the large deviation viewpoint). Instead, we will focus on the structure and computation of the optimal solution.

\noindent\textbf{Contributions.}
For the LEQG problem \eqref{LEQGcovariancesteeringproblem}, we

\begin{itemize}

\item deduce the optimal solution structure, and explain how this structure departs from the existing results known for the risk-neutral case (Sec. \ref{sec:MainResults}),

\item establish existence-uniqueness of the optimal solution in the neighborhood of the risk-neutral optimal solution (Theorem \ref{thm:LocalExistenceUniqueness} and Sec. \ref{sec:Proofs}),

\item implement a numerical procedure that builds on the structural results deduced here, and illustrate the solution via a noisy double integrator example (Sec. \ref{sec:Numerics}). 

\end{itemize}

%\noindent\textbf{Organization.}

\noindent\textbf{Notations and acronyms.} For natural number $n>1$, we use $\mathbb{S}^{n}$ to denote the set of $n\times n$ real symmetric matrices, and $\mathbb{S}^{n}_{++}$ to denote its subset comprising of positive definite matrices. The notation $\lambda_i$ denotes the $i$th eigenvalue, and $\otimes$ denotes the Kronecker product. We use the acronyms IVP and BVP for initial and boundary value problem, respectively. We abbreviate ``state transition matrix" as STM, and ``with respect to" as w.r.t.

%%%%%%%%%%%%%%%%%%%%%%%%%%%%%%%%%%

\section{Main Results}\label{sec:MainResults}
\subsection{Necessary Conditions for Optimality}\label{subsec:ConditionsForOptimality}
\begin{proposition}\label{prop:PiHODEBVP}
Consider problem \eqref{LEQGcovariancesteeringproblem} with fixed finite $\theta\in\mathbb{R}$ (equivalently problem \eqref{RiskSensitiveSB}) and assumptions A1-A3. The optimal controller is 
\begin{align}
\gamma_{\mathrm{opt}}(t,x) = -R_t^{-1}B_t^{\top}\Pi_{t\mid\theta}\: x,
\label{OptimalControl}    
\end{align}
where $\left(H_{t\mid\theta},\Pi_{t\mid\theta}\right)$ jointly solve\footnote{The subscript ${t\mid\theta}$ is meant to signify the parametric dependence of the time varying matrix on the given risk-sensitivity parameter $\theta\in\mathbb{R}$.} the coupled Riccati-Riccati matrix BVP:
\begin{subequations}
\begin{align}
-\dot{H}_{t\mid\theta} =& A^{\top}_t H_{t\mid\theta} + H_{t\mid\theta} A_t + H_{t\mid\theta} B_t R_t^{-1} B_t^{\top} H_{t\mid\theta} -Q_t\nonumber\\
&-\!\left(\Pi_{t\mid\theta} +H_{t\mid\theta}\right)\left(B_t R_t^{-1} B_t^{\top}-F_t F^{\top}_t\right)\left(\Pi_{t\mid\theta} +H_{t\mid\theta}\right)\nonumber\\
& - \colorbox{MidnightBlue!10}{$\theta\Pi_{t\mid\theta} F_t F_t^{\top} \Pi_{t\mid\theta}$},\label{Hdot}\\
-\dot{\Pi}_{t\mid\theta} =& A_t^{\top}\Pi_{t\mid\theta} + \Pi_{t\mid\theta}A_t + Q_t\nonumber\\
&- \Pi_{t\mid\theta}\left(B_t R_t^{-1}B_t^{\top}-\colorbox{MidnightBlue!10}{$\theta F_t F_t^{\top}$}\right)\Pi_{t\mid\theta}, \label{Pidot}\\
\Sigma_{0}^{-1} =& \Pi_{0\mid\theta} + H_{0\mid\theta}, \quad \Sigma_{1}^{-1} = \Pi_{1\mid\theta} + H_{1\mid\theta}.
\label{InitialAndTemrinalConditions}
\end{align}
\label{RiccatiRiccatiBVP}
\end{subequations}
\end{proposition}
\vspace*{-0.3in}
\begin{proof} 
That the optimal feedback policy is a linear state feedback follows from the standard linear-quadratic theory arguments \cite{jacobson1973optimal,duncan2013linear,chen2015optimalPartI,chen2018optimal}. In particular, the necessary conditions of optimality for \eqref{LEQGcovariancesteeringproblem} yields an ODE BVP in the covariance state-costate pair $(\Sigma_{t\mid\theta},\Pi_{t\mid\theta})\in\mathbb{S}^{n}_{++}\times\mathbb{S}^{n}$, given by
\begin{subequations}
\begin{align}
\dot{\Sigma}_{t\mid\theta} =& \left(\!A_t - B_t R_t^{-1} B_t^{\top}\Pi_{t\mid\theta}\!\right)\!\Sigma_{t\mid\theta} \nonumber\\
&+ \Sigma_{t\mid\theta}\left(\!A_t - B_t R_t^{-1} B_t^{\top}\Pi_{t\mid\theta}\!\right)^{\!\!\top} + F_t F_t^{\top}, \label{SigmadotFromPMP}\\
-\dot{\Pi}_{t\mid\theta} =& A_t^{\top}\Pi_{t\mid\theta} + \Pi_{t\mid\theta}A_t + Q_t\nonumber\\
&- \Pi_{t\mid\theta}\left(B_t R_t^{-1}B_t^{\top}-\theta F_t F_t^{\top}\right)\Pi_{t\mid\theta}, \label{PidotFromPMP}\\
&\hspace*{-0.4in}\Sigma_{0\mid\theta}=\Sigma_0\succ 0\;\text{given}, \quad \Sigma_{1\mid\theta}=\Sigma_1\succ 0\;\text{given}.\label{BCfromPMP}
\end{align}
\label{StateCostateODEs}   
\end{subequations}
The optimal controller \eqref{OptimalControl} follows from the Pontryagin's minimum principle.

Notice that \eqref{SigmadotFromPMP}-\eqref{PidotFromPMP} comprise a one-way coupled system of Lyapunov-Riccati ODEs subject to boundary conditions \eqref{BCfromPMP} only in covariance. We use the following change-of-variable $\Sigma_{t\mid\theta} \in\mathbb{S}^{n}_{++}\to H_{t\mid\theta} \in\mathbb{S}^{n}$ proposed in \cite{chen2015optimalPartI}:
\begin{align}
H_{t\mid\theta} := \Sigma_{t\mid\theta}^{-1} - \Pi_{t\mid\theta},
\label{SigmaToH} 
\end{align}
to rewrite the coupled Lyapunov-Riccati BVP \eqref{StateCostateODEs} in indeterminate pair $\left(\Sigma_{t\mid\theta},\Pi_{t\mid\theta}\right)$ as the Riccati-Riccati BVP \eqref{RiccatiRiccatiBVP} in indeterminate pair $\left(H_{t\mid\theta},\Pi_{t\mid\theta}\right)$. 
\end{proof}
\noindent In the risk neutral case ($\theta=0$), the highlighted terms in \eqref{Hdot}-\eqref{Pidot} vanish. If in addition, the noise and input channels coincide, i.e.,
\begin{align}
F_t F^{\top}_t = B_t R_t^{-1} B_t^{\top},
\label{SameChannelAssumption}    
\end{align}
then \eqref{RiccatiRiccatiBVP} reduces to the boundary value problem in \cite[eq. (2)]{chen2018optimal}. The system of equations \eqref{OptimalControl}-\eqref{RiccatiRiccatiBVP} can be seen as a risk-sensitive generalization of the same in prior work \cite{chen2018optimal}. 

To isolate the complexities due to the noise and input channel mismatch\footnote{studied in \cite{chen2015optimalPartII} for covariance steering, and more recently in \cite{bondar2026nonlinear} for distribution steering in general.} from that of risk-sensitivity, in the remaining of this letter, we consider the version of \eqref{Hdot}-\eqref{Pidot} together with \eqref{SameChannelAssumption}, i.e.,
\begin{subequations}
\begin{align}
-\dot{H}_{t\mid\theta} =& A^{\top}_t H_{t\mid\theta} + H_{t\mid\theta} A_t + H_{t\mid\theta} B_t R_t^{-1} B_t^{\top} H_{t\mid\theta} -Q_t\nonumber\\
& - \colorbox{MidnightBlue!10}{$\theta\Pi_{t\mid\theta} B_t R_t^{-1} B_t^{\top} \Pi_{t\mid\theta}$},\label{HdotSimplified}\\
-\dot{\Pi}_{t\mid\theta} =& A_t^{\top}\Pi_{t\mid\theta} + \Pi_{t\mid\theta}A_t + Q_t\nonumber\\
&+ \colorbox{MidnightBlue!10}{$(\theta - 1)\Pi_{t\mid\theta}B_t R_t^{-1}B_t^{\top}\Pi_{t\mid\theta}$}, \label{PidotSimplified}
\end{align}
\label{RiccatiRiccatiODESimplified}
\end{subequations}
and the boundary conditions \eqref{InitialAndTemrinalConditions}. As earlier, the highlights in \eqref{RiccatiRiccatiODESimplified} emphasize the terms that are new compared to the risk-neutral setting. In Sec. \ref{subsec:OnTheSolution}, we will explain how these terms give rise to new mathematical challenges not seen in the risk-neutral case.

Our main contribution is to establish the existence-uniqueness of the solution of the system of equations \eqref{RiccatiRiccatiODESimplified} and \eqref{InitialAndTemrinalConditions}, in the neighborhood of the risk-neutral ($\theta=0$) solution\footnote{The existence-uniqueness of the corresponding risk-neutral solution was established in \cite[Theorem 1]{chen2018optimal}.}. To state our main result (Theorem \ref{thm:LocalExistenceUniqueness}), we need to parameterize the solution for \eqref{RiccatiRiccatiODESimplified} and \eqref{InitialAndTemrinalConditions} in terms of certain matrix $Y_{0\mid\theta} \in \mathbb{S}^{n}$ and the given risk-sensitivity parameter $\theta\in\mathbb{R}$. In Sec. \ref{subsec:OnTheSolution}, we detail this (implicit) parametrization.

\subsection{Parametric Solution Structure}\label{subsec:OnTheSolution}
Our starting point is to write the solutions of the Riccati ODEs \eqref{HdotSimplified} and \eqref{PidotSimplified} in terms of the solutions of the following linear Hamiltonian ODEs:
\begin{subequations}
\begin{align} 
\begin{bmatrix}
        \dot{\hat{X}}_{t\mid\theta} \\
        \dot{\hat{Y}}_{t\mid\theta}
    \end{bmatrix}
    \!\!
    =&
    \!\!
    \begingroup
    \setlength{\arraycolsep}{1pt}\underbrace{\begin{bmatrix}
        A_t & -B_t R_t^{-1} B^{\top}_t \\
        -(Q_t + \theta \Pi_{t\mid\theta} B_t R_t^{-1} B^{\top}_t \Pi_{t\mid\theta}) & -A_t^{\top}
    \end{bmatrix}}_{=:\hat{M}_{t\mid\theta}}
    \endgroup
   \!\!\begin{bmatrix}
         \hat{X}_{t\mid\theta} \\
         \hat{Y}_{t\mid\theta}
     \end{bmatrix},\label{defXhatYhat}\\
\begin{bmatrix}
        \dot{X}_{t\mid\theta} \\
        \dot{Y}_{t\mid\theta}
    \end{bmatrix}
    \!\!=&
    \!\underbrace{\begin{bmatrix}
        A_t & (\theta - 1) B_t R_t^{-1} B^{\top}_t \\
        -Q_t & -A_t^{\top}
    \end{bmatrix}}_{=:M_{t\mid\theta}}
    \begin{bmatrix}
        X_{t\mid\theta} \\
        Y_{t\mid\theta}
    \end{bmatrix}. \label{defXY}
\end{align}
\label{HamiltonianODEs}
\end{subequations}
The solutions of \eqref{HdotSimplified} and \eqref{PidotSimplified} admit the well-known \cite[p. 156]{brockett1970finite} linear fractional representations 
\begin{align}
H_{t\mid\theta}  = -\left(\hat{X}_{t\mid\theta}^{\top}\right)^{-1} \hat{Y}_{t\mid\theta}^{\top}, \quad \Pi_{t\mid\theta}  = Y_{t\mid\theta} X_{t\mid\theta}^{-1},
\label{LFTrepresentations}    
\end{align}
provided that $X_{t\mid\theta},\hat{X}_{t\mid\theta}$ are invertible $\forall t\in[0,1]$, $\theta\in\mathbb{R}$. 
\begin{remark}\label{Remark:MnotequaltoMhat}
Unlike the risk-neutral case in \cite{chen2018optimal}, $\hat{M}_{t\mid\theta}\neq M_{t\mid\theta}$ for $\theta\neq 0$.%This prevents us from using the symplectic identities in \cite[Lemma 3]{chen2018optimal}.
\end{remark}
For $0\leq s < t \leq 1$, let $\hat{\Phi}_{s\rightarrow t\mid\theta},\Phi_{s\rightarrow t\mid\theta}\in\mathbb{R}^{2n\times 2n}$ denote the STMs from time $s$ to $t$, associated with the coefficient matrices $\hat{M}_{t\mid\theta},M_{t\mid\theta}$, respectively, i.e.,
\begin{subequations}
\begin{align}
\frac{\partial}{\partial t}\hat{\Phi}_{s\rightarrow t \mid \theta} = \hat{M}_{t\mid\theta}\hat{\Phi}_{s\rightarrow t \mid \theta}, \quad \hat{\Phi}_{s\rightarrow s\mid \theta} = I, \label{PhihatODEivp}\\
\frac{\partial}{\partial t}\Phi_{s\rightarrow t \mid \theta} = M_{t\mid\theta}\Phi_{s\rightarrow t \mid \theta}, \quad \Phi_{s\rightarrow s\mid \theta} = I.\label{PhiODEivp}
\end{align}
\label{STMs}
\end{subequations}
Expressing the corresponding STMs in terms of their $n\times n$ blocks, for any $t\in[0,1]$, we write
% \begin{align}
% \begin{bmatrix}
% \hat{\Phi}_{11} & \hat{\Phi}_{12}\\
% \hat{\Phi}_{21} & \hat{\Phi}_{22}
% \end{bmatrix} := \hat{\Phi}_{0\rightarrow 1\mid\theta}, \quad \begin{bmatrix}
% \Phi_{11} & \Phi_{12}\\
% \Phi_{21} & \Phi_{22}
% \end{bmatrix} := \Phi_{0\rightarrow 1\mid\theta}.
% \label{STMblocks}
% \end{align}
% Then 
\begin{align*}
&\begin{bmatrix}
         \hat{X}_{t\mid\theta} \\
         \hat{Y}_{t\mid\theta}
     \end{bmatrix}\!=\!\underbrace{\begin{bmatrix}\hat{\Phi}_{0\rightarrow t\mid\theta}^{11} & \hat{\Phi}_{0\rightarrow t\mid\theta}^{12}\\
\hat{\Phi}_{0\rightarrow t\mid\theta}^{21} & \hat{\Phi}_{0\rightarrow t\mid\theta}^{22}
\end{bmatrix}}_{= \hat{\Phi}_{0\rightarrow t\mid\theta}}
\!\begin{bmatrix}
         \hat{X}_{0\mid\theta} \\
         \hat{Y}_{0\mid\theta}
     \end{bmatrix},\\
     &\begin{bmatrix}
         X_{t\mid\theta} \\
         Y_{t\mid\theta}
     \end{bmatrix}\!=\!\underbrace{\begin{bmatrix}\Phi_{0\rightarrow t\mid\theta}^{11} & \Phi_{0\rightarrow t\mid\theta}^{12}\\
\Phi_{0\rightarrow t\mid\theta}^{21} & \Phi_{0\rightarrow t\mid\theta}^{22}
\end{bmatrix}}_{=\Phi_{0\rightarrow t\mid\theta}}\!\begin{bmatrix}
         X_{0\mid\theta} \\
         Y_{0\mid\theta}
     \end{bmatrix}\!.    
\end{align*}
As in the proof of \cite[Theorem 1]{chen2018optimal}, without loss of generality, we can set $\hat{X}_{0\mid\theta}=X_{0\mid\theta}=I$ since their initial values can be absorbed into $\hat{Y}_{0\mid\theta},Y_{0\mid\theta}$ keeping $H_{0\mid\theta},\Pi_{0\mid\theta}\in\mathbb{S}^{n}$ invariant. Then from \eqref{LFTrepresentations}, both $\hat{Y}_{0\mid\theta},Y_{0\mid\theta}$ are symmetric. In particular, combining \eqref{SigmaToH} and \eqref{LFTrepresentations}, we obtain
\begin{align}
\hat{Y}_{0\mid\theta} = Y_{0\mid\theta} - \Sigma_{0}^{-1}.
\label{RelatingY0AndY0hat}    
\end{align}
Furthermore,
\begin{subequations}
\begin{align}
X_{t\mid\theta} &= \Phi_{0\rightarrow t\mid\theta}^{11} + \Phi_{0\rightarrow t\mid\theta}^{12}Y_{0\mid\theta},\label{XtFromYt}\\
Y_{t\mid\theta} &= \Phi_{0\rightarrow t\mid\theta}^{21} + \Phi_{0\rightarrow t\mid\theta}^{22}Y_{0\mid\theta},\label{YtFromY0}
\end{align}
\label{XtYtFromY0}
\end{subequations}
and hence from \eqref{LFTrepresentations}, the matrix $\Pi_{t\mid\theta}\in\mathbb{S}^{n}$ is determined by the choice of $Y_{0\mid\theta}\in\mathbb{S}^{n}$.

A notable departure from the risk-neutral case is that the matrix $\hat{M}_{t\mid\theta}$ in \eqref{defXhatYhat} depends explicitly on $\Pi_{t\mid\theta}\in\mathbb{S}^{n}$, and therefore on $Y_{0\mid\theta}$. Hence, the blocks of $\hat{\Phi}_{0\rightarrow t\mid\theta}$ should be viewed as functions of both $\theta$ and $Y_{0\mid\theta}$. In contrast, the blocks of $\Phi_{0\rightarrow t\mid\theta}$ are functions of $\theta$, but do not depend on $Y_{0\mid\theta}$. 

Particularizing the above observations for $t=1$, we find
\begin{subequations}
    \begin{align}
            &H_{1\mid\theta} = -\left(\hat{X}_{1\mid\theta}^{\top}\right)^{-1} \hat{Y}_{1\mid\theta}^{\top}=\nonumber\\
        &{\small{-\!\left(\!\!\left(\!{\hat{\Phi}}_{0\rightarrow 1\mid\theta}^{11}\right)^{\!\!\top}\!\!\!+\!\hat{Y}_{0\mid\theta}\left(\!{\hat{\Phi}}_{0\rightarrow 1\mid\theta}^{12}\right)^{\!\!\top}\!\right)^{\!\!-1}\!\!\!\!\left(\!\!\left(\!{\hat{\Phi}_{0\rightarrow 1\mid\theta}^{21}}\!\right)^{\!\!\top}\!\!\!+\! \hat{Y}_{0\mid\theta}\!\left(\!{\hat{\Phi}_{0\rightarrow 1\mid\theta}^{22}}\!\right)^{\!\!\!\top}\!\right)\!}},\label{H1}\\
        &\Pi_{1\mid\theta} = Y_{1\mid\theta} X_{1\mid\theta}^{-1} \nonumber\\
        &= \!\left(\!\Phi_{0\rightarrow 1\mid\theta}^{21} + \Phi_{0\rightarrow 1\mid\theta}^{22} Y_{0\mid\theta}\!\right)\!\!\left(\!\Phi_{0\rightarrow 1\mid\theta}^{11} + \Phi_{0\rightarrow 1\mid\theta}^{12}Y_{0\mid\theta}\!\right)^{\!\!-1}\!\!.\label{Pi1}
    \end{align}
\label{Pi1AndH1}    
\end{subequations}
Since \eqref{RelatingY0AndY0hat} allows us to write $\hat{Y}_{0\mid\theta}$ in terms of $Y_{0\mid\theta}$, we view \eqref{H1}-\eqref{Pi1} as functions of $\theta$ and $Y_{0\mid\theta}$. Therefore, using \eqref{SigmaToH}, we arrive at an implicit algebraic equation in indeterminate $Y_{0\mid\theta}$ given by
\begin{align}
G(\theta,Y_{0\mid\theta}) := \Pi_{1\mid\theta} + H_{1\mid\theta} - \Sigma_{1}^{-1} = 0.
\label{defG}
\end{align}
If a matrix $Y_{0\mid\theta}\in\mathbb{S}^{n}$ solving \eqref{defG} exists and is unique in a suitable sense, then it determines $\Pi_{t\mid\theta}$ (via \eqref{LFTrepresentations}), and thus the optimal control \eqref{OptimalControl}. However, an explicit solution of \eqref{defG} in the form $Y_{0\mid\theta}=Y_{0\mid\theta}(\theta)$ is not available in general.

In the risk-neutral ($\theta=0$) case, $\hat{M}_{t\mid 0}= M_{t\mid 0}$, so the corresponding STMs match\footnote{$X_{t\mid 0}\neq \hat{X}_{t\mid 0}$ in general, since they solve the same Hamiltonian ODE but with different initial conditions. This makes Lemma \ref{lemma:PartialJacobianOfGAtThetaEqualToZero} in Sec. \ref{sec:Proofs} nontrivial.}, i.e., $\hat{\Phi}_{s\rightarrow t\mid 0}=\Phi_{s\rightarrow t\mid 0}$. This STM matching plays a crucial role in the solution derived in \cite[Theorem 1]{chen2018optimal}. In that case, combining the terminal boundary condition $\Sigma_{1}^{-1}=\Pi_{1\mid0} + H_{1\mid0}$ with \eqref{Pi1AndH1}, and using \eqref{RelatingY0AndY0hat}, the symplectic identities in \cite[Lemma 3]{chen2018optimal} turn \eqref{defG} into a  quadratic equation in unknown $Y_{0\mid 0}$. In the proof for \cite[Theorem 1]{chen2018optimal}, the two roots ($Y_{0\mid 0}^{+},Y_{0\mid 0}^{-}$) of this quadratic equation found via completion-of-square, were 
\begin{align}
&Y_{0\mid 0}^{\pm} = \frac{1}{2}\Sigma_0^{-1}-\left(\Phi^{12}_{0\rightarrow 1}\right)^{-1}\Phi^{11}_{0\rightarrow 1}\nonumber\\
&\pm \Sigma_{0}^{-\frac{1}{2}}\left(\frac{1}{4}I + \Sigma_{0}^{\frac{1}{2}}\left(\Phi^{12}_{0\rightarrow 1}\right)^{-1}\Sigma_{1}\left(\Phi^{12}_{0\rightarrow 1}\right)^{-\top}\Sigma_{0}^{\frac{1}{2}}\right)^{\frac{1}{2}}\Sigma_{0}^{-\frac{1}{2}}.
\label{Y0given0PlusMinus}   
\end{align}
In that work, it was shown that only one of these roots, namely $Y_{0\mid 0}^{-}$, 
ensures invertibility\footnote{Their lack of invertibility for $Y_{0\mid 0}^{+}$ implies finite escape times for $\Pi_{t\mid 0}$,$H_{t\mid 0}$ for some $t\in[0,1]$, i.e., inadmissible initial condition for \eqref{RiccatiRiccatiODESimplified} at $\theta=0$.} of $X_{t\mid 0},\hat{X}_{t\mid 0}$ $\forall t\in[0,1]$. By \eqref{LFTrepresentations}, this in turn ensures the unique admissible solution of \eqref{RiccatiRiccatiODESimplified} for $\theta=0$.

In our $\theta \neq 0$ case, per Remark \ref{Remark:MnotequaltoMhat}, $\hat{\Phi}_{s\rightarrow t\mid\theta}\neq\Phi_{s\rightarrow t\mid\theta}$, and the \emph{mixed} STM terms no longer satisfy the same symplectic identities. Consequently, we have to contend with the implicit algebraic equation \eqref{defG} in $Y_{0\mid\theta}$. Now we are ready to state our main result.
\begin{theorem}\label{thm:LocalExistenceUniqueness}
Consider problem \eqref{LEQGcovariancesteeringproblem} with fixed finite $\theta\in\mathbb{R}$, assumptions A1-A3, and the matched channel condition \eqref{SameChannelAssumption}. Let $Y_{0\mid0}^{-}$ denote the unique admissible solution of \eqref{defG} in the risk-neutral LQG covariance steering problem. Then there exist open sets $\mathcal{I}\subset\mathbb{R}$, $\mathcal{U}\subset\mathbb{S}^{n}$ such that $\left(0,Y_{0\mid0}^{-}\right)\in \mathcal{I}\times \mathcal{U}$, and a unique $\mathcal{C}^{1}$ mapping $h:\mathcal{I}\rightarrow\mathcal{U}$ such that $h(0)=Y_{0\mid0}^{-}$ and $G(\theta,h(\theta))=0$ $\forall\theta\in\mathcal{I}$. In other words, the LEQG problem \eqref{LEQGcovariancesteeringproblem} has a local $\mathcal{C}^1$ solution branch\footnote{For $\theta\in\mathcal{I}$, both $\hat{X}_{t\mid\theta},X_{t\mid\theta}$ remain invertible $\forall t\in[0,1]$ by continuity (determinants are continuous in $\theta$, and by \cite[Thm. 1]{chen2018optimal}, invertibility holds at $\theta=0$).} through $Y_{0\mid0}^{-}$.
\end{theorem}
\noindent In Section \ref{sec:Proofs} that follows, we prove Theorem \ref{thm:LocalExistenceUniqueness} using several auxiliary results (Proposition \ref{prop:G_C1}, Lemma \ref{lemma:PartialJacobianOfGAtThetaEqualToZero}, Proposition \ref{prop:NonsingularityOfPartialJacobianOfGAtThetaEqualToZero}) derived therein.

%%%%%%%%%%%%%%%%%%%%%%%%%%%%%%%%%%

\section{Proofs and Auxiliary Results}\label{sec:Proofs}
Our results in this Section are organized and interconnected as in the block diagram below.
\begin{figure}[h!]
\centering
\includegraphics[width=\linewidth]{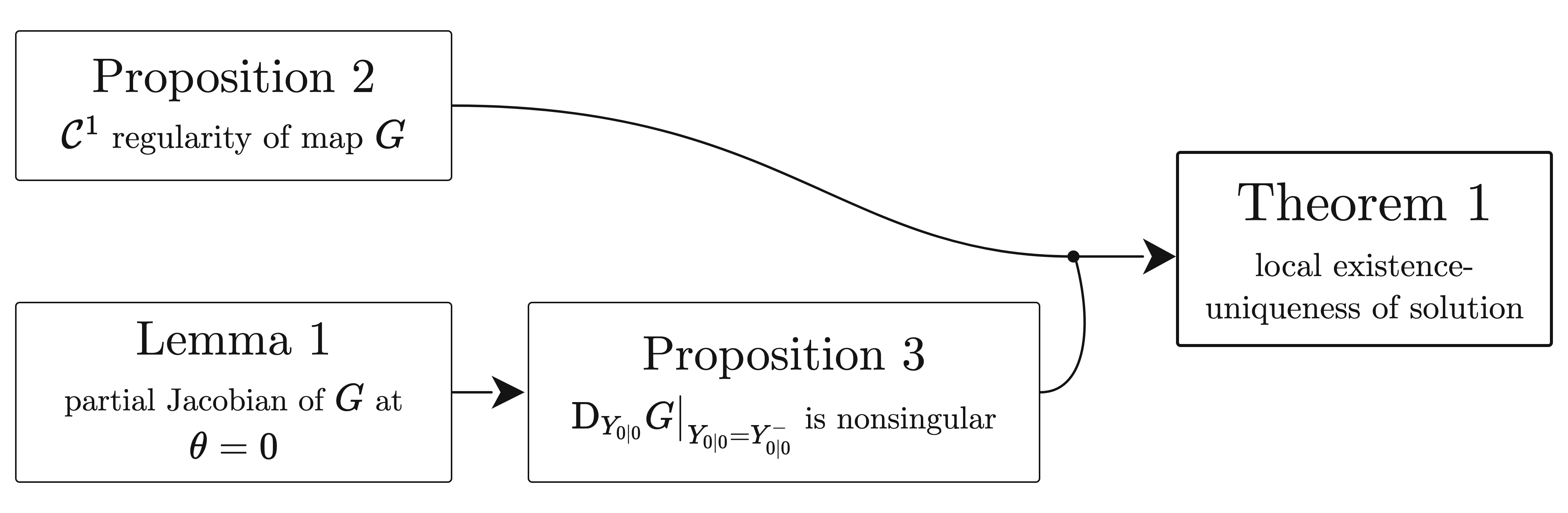}
\end{figure}
\begin{proposition} \label{prop:G_C1}
Consider $G$ in \eqref{defG} as a nonlinear map $G(\theta,Y_{0\mid\theta}):\mathcal{I}\times\mathcal{U}\mapsto\mathcal{V}$ where $\mathcal{I}\times\mathcal{U}$ and $\mathcal{V}$ are bounded open subsets of $\mathbb{R}\times\mathbb{S}^{n}$ and $\mathbb{S}^{n}$, respectively. Then $G\in \mathcal{C}^1(\mathcal{I}\times\mathcal U;\mathcal{V})$.
\end{proposition}

\begin{proof}
We establish the desired regularity of $G$ w.r.t. the tuple $\left(\theta,Y_{0\mid\theta}\right)$ by following the dependence structure of the quantities entering the definition of $G$:
\[
\left(\theta,Y_{0\mid\theta}\right)\!
\longmapsto\!
\Pi_{t\mid\theta}\left(\theta,Y_{0\mid\theta}\right)
\!\longmapsto\!
H_{t\mid\theta}\left(\theta,Y_{0\mid\theta}\right)
\!\longmapsto\!
G(\theta,Y_{0\mid\theta}).
\]
All matrix spaces involved are finite-dimensional and may therefore be identified with Euclidean spaces.%, so that the standard differentiability results for ODEs apply. 

Consider the Riccati ODE \eqref{PidotSimplified} in indeterminate $\Pi_{t\mid\theta}$ with initial condition $\Pi_{0\mid\theta}=Y_{0\mid\theta}$, where $\theta$ serves as a parameter. The corresponding vector field is continuous in $t$, and $\mathcal{C}^{1}$ w.r.t. $(\Pi_{t\mid\theta},\theta)$ since it comprises of continuous and bounded in time $A_t,B_t,Q_t$ (Assumptions A1-A2), together with finite sums and products involving $\Pi_{t\mid\theta}$ and $\theta$. Therefore, by the smoothness theorem for parameter-dependent ODE IVP \cite[Chapter V, Theorem 3.1]{hartman2002ordinary}, $\Pi_{t\mid\theta}$ is jointly $\mathcal{C}^1$ w.r.t. $(\theta,Y_{0\mid\theta})$.

Next, consider the Riccati ODE \eqref{HdotSimplified} in indeterminate $H_{t\mid\theta}$ with initial condition (see \eqref{SigmaToH}) $H_{0\mid\theta}=\Sigma_0^{-1}-Y_{0\mid\theta}$ (which is affine, thus $\mathcal{C}^1$ in $Y_{0\mid\theta}$). Since $\Pi_{t\mid\theta}$ is jointly $\mathcal{C}^1$ w.r.t. $(\theta,Y_{0\mid\theta})$, the vector field of \eqref{HdotSimplified} is continuous in $t$, and $\mathcal{C}^{1}$ w.r.t. $(H_{t\mid\theta},\theta,Y_{0\mid\theta})$. By the same smoothness theorem \cite[Chapter V, Theorem 3.1]{hartman2002ordinary}, $H_{t\mid\theta}$ is jointly $\mathcal{C}^1$ w.r.t. $(\theta,Y_{0\mid\theta})$.

Since evaluation at the fixed terminal time $t=1$ preserves $\mathcal{C}^1$-regularity, both $\Pi_{1\mid\theta},H_{1\mid\theta}$ are $\mathcal{C}^1$ in $(\theta,Y_{0\mid\theta})$. The matrix $\Sigma_1^{-1}$ is constant. Therefore, $G\in \mathcal{C}^1(\mathcal{I}\times\mathcal U;\mathcal{V})$. 
\end{proof}

\begin{lemma}\label{lemma:PartialJacobianOfGAtThetaEqualToZero}
At $\theta=0$, the partial Jacobian
\begin{align}
\Jac_{Y_{0\mid 0}}G = X_{1\mid 0}^{-\top} \otimes X_{1\mid 0}^{-\top} - \hat{X}_{1\mid 0}^{-\top} \otimes \hat{X}_{1\mid 0}^{-\top}.
\label{PartialJacobian}
\end{align}
\end{lemma}
\begin{proof}
From \eqref{defG}, we have 
\begin{align}
\differential G(0,Y_{0\mid 0}) = \differential\Pi_{1\mid 0} + \differential H_{1\mid 0}.
\label{dGEqualsdPiPlusdH}    
\end{align}
Using \eqref{Pi1}, 
\begin{align}
&\differential\Pi_{1\mid 0} = \differential\!\left(Y_{1\mid 0}X_{1\mid 0}^{-1}\!\right)\! = (\differential Y_{1\mid 0}) X_{1\mid 0}^{-1} - Y_{1\mid 0} X_{1\mid 0}^{-1}(\differential X_{1\mid 0}) X_{1\mid 0}^{-1}\nonumber\\
&=(\Phi^{22}_{0\rightarrow 1\mid 0} \differential Y_{0\mid 0}) X_{1\mid 0}^{-1} - \Pi_{1\mid 0} (\Phi^{12}_{0\rightarrow 1\mid 0} \differential Y_{0\mid 0}) X_{1\mid 0}^{-1}\nonumber\\
&= \left(\Phi^{22}_{0\rightarrow 1\mid 0}-\Pi_{1\mid 0}\Phi^{12}_{0\rightarrow 1\mid 0}\right)\left(\differential Y_{0\mid 0}\right)X_{1\mid 0}^{-1}.
\label{differnetialPiIntermediate}
\end{align}
Since $\Pi_{1\mid 0}\in\mathbb{S}^{n}$, we have $\Pi_{1\mid0} = \Pi_{1\mid 0}^\top = X_{1\mid 0}^{-\top} Y_{1\mid 0}^{\top}$. Let $A := \Phi^{22}_{0\rightarrow 1\mid 0}-\Pi_{1\mid 0}\Phi^{12}_{0\rightarrow 1\mid 0} = \Phi^{22}_{0\rightarrow 1\mid 0}- X_{1\mid 0}^{-\top} Y_{1\mid 0}^{\top}\Phi^{12}_{0\rightarrow 1\mid 0}$. Left-multiplying $A$ by $X_{1\mid 0}^\top=\left(\!\Phi_{0\rightarrow 1 \mid 0}^{11}\!\right)^{\!\!\top} \!+\! Y_{0\mid 0} \!\left(\!\Phi_{0\rightarrow 1 \mid 0}^{12}\!\right)^{\!\top}$, we obtain
\begin{align*}
X_{1\mid 0}^{\!\top} A
&= \left(\!\Phi^{11}_{0\rightarrow 1 \mid 0}\!\right)^{\!\top}\!\! \Phi_{0\rightarrow 1 \mid 0}^{22} + Y_{0\mid 0} \left(\!\Phi_{0\rightarrow 1 \mid 0}^{12}\!\right)^{\!\top} \Phi_{0\rightarrow 1 \mid 0}^{22} \nonumber\\
&\quad- \left(\!\Phi_{0\rightarrow 1 \mid 0}^{21}\!\right)^{\!\top} \!\!\Phi_{0\rightarrow 1 \mid 0}^{12} - Y_{0\mid 0} \left(\!\Phi_{0\rightarrow 1 \mid 0}^{22}\!\right)^{\!\top} \Phi_{0\rightarrow 1 \mid 0}^{12} \nonumber \\
&= \left(\!\Phi^{11}_{0\rightarrow 1 \mid 0}\!\right)^{\!\top}\!\! \Phi_{0\rightarrow 1 \mid 0}^{22}- \left(\!\Phi_{0\rightarrow 1 \mid 0}^{21}\!\right)^{\!\top} \!\!\Phi_{0\rightarrow 1 \mid 0}^{12} = I,
\end{align*}
where the second equality is due to the symplectic identity \cite[eq. (7b)]{chen2018optimal}, and the third is due to the symplectic identity \cite[eq. (7a)]{chen2018optimal}. Thus\footnote{This also shows $X_{1\mid 0}$ is nonsingular.} $A = X_{1\mid 0}^{-\top}$, which simplifies \eqref{differnetialPi} to
\begin{align}
\differential\Pi_{1\mid 0} = X_{1\mid 0}^{-\top}\left(\differential Y_{0\mid 0}\right)X_{1\mid 0}^{-1}.
\label{differnetialPi}
\end{align}
Starting from \eqref{H1}, similar computation\footnote{This also shows that $\hat{X}_{1\mid 0}$ is nonsingular.} yields %{\red{note to Abhishek: expand if space permits}}
\begin{align}
\differential H_{1\mid 0} = -\hat{X}_{1\mid 0}^{-\top} \left(\differential Y_{0\mid 0}\right) \hat{X}_{1\mid 0}^{-1}.
\label{differentialH}
\end{align}
Substituting $\differential\Pi_{1\mid 0}$ and $\differential H_{1\mid 0}$ from \eqref{differnetialPi} and \eqref{differentialH} back into \eqref{dGEqualsdPiPlusdH}, we obtain 
\begin{align}
\differential G(0,Y_{0\mid 0}) = X_{1\mid 0}^{-\top}\left(\!\differential Y_{0\mid 0}\!\right)X_{1\mid 0}^{-1} -\hat{X}_{1\mid 0}^{-\top} \!\left(\!\differential Y_{0\mid 0}\!\right)\!\hat{X}_{1\mid 0}^{-1}.
\label{BeforeApplyingVec}
\end{align}
Applying the vectorization (${\mathrm{vec}}$) operator to both sides of \eqref{BeforeApplyingVec}, and using the identity ${\mathrm{vec}}(PQR)=(R^{\top}\otimes P){\mathrm{vec}}(Q)$, we get
\begin{align}
&{\mathrm{vec}}\:\differential G(0,Y_{0\mid 0})\nonumber\\
&= \left[\left(X_{1\mid 0}^{-\top} \otimes X_{1\mid 0}^{-\top}\right) - \left(\hat{X}_{1\mid 0}^{-\top} \otimes \hat{X}_{1\mid 0}^{-\top}\right)\right]{\mathrm{vec}}\:\differential Y_{0\mid 0}.
\label{AfterApplyingVec}
\end{align}
From \eqref{AfterApplyingVec}, using the Jacobian identification rule \cite[p. 199]{magnus2007matrix}, we arrive at \eqref{PartialJacobian}. 
\end{proof}

\begin{remark}\label{remark:PartialJacobianAtnonzerotheta}
A tedious computation shows that at $\theta\neq 0$,
\begin{align}
\Jac_{Y_{0\mid \theta}}G = &X_{1\mid \theta}^{-\top} \otimes X_{1\mid \theta}^{-\top} - \hat{X}_{1\mid \theta}^{-\top} \otimes \hat{X}_{1\mid \theta}^{-\top}\nonumber\\
&+\theta\int_{0}^{1}\left(U_{s\mid\theta}\otimes V_{s\mid\theta}
+
V_{s\mid\theta}\otimes U_{s\mid\theta}\right)\differential s,
\label{PartialJacobianAtthetaNotEqualToZero}    
\end{align}
where
\begin{subequations}
\begin{align}
U_{s\mid\theta} &:=
\hat{X}_{1\mid\theta}^{-\top}\hat{X}_{s\mid\theta}^{\top} X_{s\mid\theta}^{-\top}, \label{defUs}\\
V_{s\mid\theta} &:=\hat{X}_{1\mid\theta}^{-\top}\hat{X}_{s\mid\theta}^{\top}
\Pi_{s\mid\theta} B_s R_s^{-1} B_s^\top X_{s\mid\theta}^{-\top}. \label{defVs}
\end{align}
\label{defUsVs}    
\end{subequations}
While \eqref{PartialJacobianAtthetaNotEqualToZero} generalizes \eqref{PartialJacobian}, we will not use this result in the proofs that follow.
\end{remark}

\begin{proposition}\label{prop:NonsingularityOfPartialJacobianOfGAtThetaEqualToZero}
The partial Jacobian $\Jac_{Y_{0\mid 0}}G$ evaluated at $Y_{0\mid 0}^{-}$ in \eqref{Y0given0PlusMinus}, is nonsingular.
\end{proposition}
\begin{proof}
For convenience, let
\begin{subequations}
\begin{align}
&\Psi_{01}:=\nonumber\\
&\Sigma_{0}^{-\frac{1}{2}}\!\left(\!\frac{1}{4}I + \Sigma_{0}^{\frac{1}{2}}\!\left(\!\Phi^{12}_{0\rightarrow 1}\!\right)^{-1}\Sigma_{1}\left(\!\Phi^{12}_{0\rightarrow 1}\!\right)^{-\top}\Sigma_{0}^{\frac{1}{2}}\!\right)^{\!\frac{1}{2}}\!\Sigma_{0}^{-\frac{1}{2}},
\label{defTheta}\\
&S_0 := \frac{1}{2}\Sigma_0^{-1}\label{defS0}.
\end{align}
\label{defPsiS0}    
\end{subequations}
As explained in Sec. \ref{subsec:OnTheSolution}, $\hat{X}_{0\mid\theta}=X_{0\mid\theta}=I$ $\forall\theta\in[0,1]$, so the matrices $\hat{X}_{1\mid 0}=\Phi^{11}_{0\rightarrow 1\mid 0} + \Phi^{12}_{0\rightarrow 1\mid 0}  \hat{Y}_{0\mid 0}^{-}$ and $X_{1\mid 0}=\Phi^{11}_{0\rightarrow 1\mid 0} + \Phi^{12}_{0\rightarrow 1\mid 0}  Y_{0\mid 0}^{-}$ simplify to
\begin{subequations}
\begin{align}
\hat{X}_{1\mid 0} &= X_{1\mid 0} - \Phi^{12}_{0\rightarrow 1\mid 0}\Sigma_0^{-1} = -\Phi^{12}_{0\rightarrow 1\mid 0}\!\left(\Psi_{01} + S_0\right),\label{SimplifiedX1hat}\\
X_{1\mid 0} &= -\Phi^{12}_{0\rightarrow 1\mid 0}\left(\Psi_{01}-S_0\right),\label{SimplifiedX1}
\end{align}
\label{SimplifiedX1andX1hat}\end{subequations}
where we have used \eqref{RelatingY0AndY0hat} and \eqref{Y0given0PlusMinus}.

Per Assumption A3, both $\Psi_{01},S_0$ in \eqref{defPsiS0} are strictly positive definite. In particular, the positive definiteness of $\Psi_{01}$ follows from that 
\begin{itemize}

\item a positive definite matrix has unique (principal) positive definite square root,

\vspace*{0.05in}

\item $\Phi^{12}_{0\rightarrow 1}$ is invertible \cite[Lemma 3]{chen2018optimal},

\vspace*{0.05in}

\item the matrix $\Sigma_{0}^{\frac{1}{2}}\!\left(\!\Phi^{12}_{0\rightarrow 1}\!\right)^{-1}\Sigma_{1}\left(\!\Phi^{12}_{0\rightarrow 1}\!\right)^{-\top}\Sigma_{0}^{\frac{1}{2}}\succ 0$ since it is congruent \cite[Ch. 4.5]{horn2012matrix} to $\Sigma_1\succ 0$ via congruence transform defined by invertible matrix $\Sigma_{0}^{\frac{1}{2}}\!\left(\!\Phi^{12}_{0\rightarrow 1}\right)^{\!-1}$.  
\end{itemize}
Hence, $\Psi_{01} + S_0 \succ 0$. Furthermore, 
$$\lambda_i\!\left(\!\frac{1}{4}I + \Sigma_{0}^{\frac{1}{2}}\!\left(\!\Phi^{12}_{0\rightarrow 1}\!\right)^{-1}\Sigma_{1}\left(\!\Phi^{12}_{0\rightarrow 1}\!\right)^{-\top}\Sigma_{0}^{\frac{1}{2}}\!\right)\!>\frac{1}{4}\quad\forall i\in\llbracket n\rrbracket$$
implies that
\begin{align}
&\Psi_{01}-S_0=\nonumber\\
&\Sigma_0^{-\frac{1}{2}}\!\!\left(\!\!\left(\!\frac{1}{4}I + \Sigma_{0}^{\frac{1}{2}}\!\!\left(\!\Phi^{12}_{0\rightarrow 1}\!\right)^{-1}\Sigma_{1}\left(\!\Phi^{12}_{0\rightarrow 1}\!\right)^{-\top}\Sigma_{0}^{\frac{1}{2}}\!\right)^{\!\frac{1}{2}} \!\!- \frac{1}{2}I\!\right)\!\Sigma_0^{-\frac{1}{2}}\nonumber\\
&\succ 0.\label{PsiMinusSPosDef}
\end{align}
From \eqref{SimplifiedX1andX1hat}, taking the inverse transposes, we get
\begin{subequations}
\begin{align}
\hat{X}_{1\mid 0}^{-\top} &= -\left(\Phi^{12}_{0\rightarrow 1\mid 0}\right)^{-\top}\!\left(\Psi_{01} + S_0\right)^{-1},\label{InverseTransposeX1hat}\\
X_{1\mid 0}^{-\top} &= -\left(\Phi^{12}_{0\rightarrow 1\mid 0}\right)^{-\top}\!\left(\Psi_{01}-S_0\right)^{-1}.\label{InverseTransposeX1}
\end{align}
\label{InverseTransposeOfX1andX1hat}\end{subequations}
Substituting \eqref{InverseTransposeOfX1andX1hat} in \eqref{PartialJacobian} from Lemma \ref{lemma:PartialJacobianOfGAtThetaEqualToZero}, and using Kronecker product properties, we obtain
\begin{align}
\Jac_{Y_{0\mid 0}}G\bigg\vert_{Y_{0\mid 0} = Y_{0\mid 0}^{-}} = \!\left(\!\left(\Phi^{12}_{0\rightarrow 1}\right)^{-\top}\!\otimes\!\left(\Phi^{12}_{0\rightarrow 1}\right)^{-\top}\right)\Omega,
\label{JacobianOfGAtthetaEqualsZero}
\end{align}
where
\begin{align}
\Omega:=&\left((\Psi_{01} - S_0)^{-1} \otimes (\Psi_{01} - S_0)^{-1}\right) \nonumber\\
&- \left((\Psi_{01} + S_0)^{-1} \otimes (\Psi_{01} + S_0)^{-1} \right).
\label{defOmega}    
\end{align}
Since $\Phi^{12}_{0\rightarrow 1}$ is invertible, \eqref{JacobianOfGAtthetaEqualsZero} is nonsingular if and only if $\det\:\Omega \neq 0$.

Since $S_0 \succ 0$ and $\Psi_{01}$ are real symmetric, by the simultaneous diagonalizability theorem \cite[Theorem 7.6.4(a)]{horn2012matrix}, \cite[Theorem 20.1]{prasolov1994problems}, there exists nonsingular $P$ such that
\begin{align}
P^\top S_0 P = I, \quad P^\top \Psi_{01} P= \operatorname{diag}(\psi_1, \dots, \psi_n),
\label{SimulatenousDiagnoalizability}    
\end{align}
where $\psi_i>0$ $\forall i\in\llbracket n\rrbracket$ denote the $i$th eigenvalue of $\Psi_{01}\succ 0$.

Since $\Psi_{01} - S_0 \succ 0$ (see \eqref{PsiMinusSPosDef}), the simultaneous diagonalizability \eqref{SimulatenousDiagnoalizability} implies the stronger: $\psi_i > 1$ $\forall i\in\{1,\hdots,n\}$. Then the eigenvalues of $\left(\Psi_{01} - S_0\right)^{-1}$ and $\left(\Psi_{01} + S_0\right)^{-1}$ are $\frac{1}{\psi_i - 1}$ and $\frac{1}{\psi_i + 1}$, respectively. So the spectrum of $\Omega$ in \eqref{defOmega} comprises of
\begin{align*}
&\frac{1}{(\psi_i - 1)(\psi_j - 1)} - \frac{1}{(\psi_i + 1)(\psi_j + 1)}\nonumber\\
=&\frac{2(\psi_i + \psi_j)}{(\psi_i^2 - 1)(\psi_j^2 - 1)} \qquad\qquad\qquad \forall (i,j) \in \llbracket n \rrbracket \times\llbracket n \rrbracket,
\end{align*}
which are positive because $\psi_i > 1$ $\forall i\in\llbracket n\rrbracket$. This proves $\Omega\succ 0$, and in particular $\det\:\Omega\neq 0$. Hence \eqref{JacobianOfGAtthetaEqualsZero} is nonsingular.
\end{proof}

\begin{proof}[Proof of Theorem 1]
Specializing Proposition \ref{prop:G_C1} at $Y_{0\mid 0} = Y_{0\mid 0}^{-}$, the mapping $G$ is $\mathcal{C}^{1}$ in a neighborhood of $(0,Y_{0\mid 0}^{-})$. Moreover, since $Y_{0\mid 0}^{-}$ is the unique admissible risk-neutral solution, from \eqref{defG}, we have that $G(0,Y_{0\mid 0}^{-})=0$. 

By Proposition \ref{prop:NonsingularityOfPartialJacobianOfGAtThetaEqualToZero}, the partial Jacobian $\Jac_{Y_{0\mid 0}}G$
evaluated at $Y_{0\mid 0}^{-}$, is nonsingular. Hence, the implicit function theorem \cite[Theorem 3.4]{edwards1994advanced} applies at $\left(0,Y_{0\mid 0}^{-}\right)$. Therefore, there exist an open interval $\mathcal{I}\subset\mathbb{R}$ containing $\theta=0$, a neighborhood $\mathcal{U}\subset\mathbb{S}_n$ containing $Y_{0\mid 0}$, and a unique $\mathcal{C}^1$ mapping $h:\mathcal{I}\rightarrow\mathcal{U}$ such that $h(0)=Y_{0\mid 0}^{-}$ and $G(\theta,h(\theta))=0$ $\forall\theta\in\mathcal{I}$. This completes the proof.
\end{proof}
%%%%%%%%%%%%%%%%%%%%%%%%%%%%%%%%%%

\section{Numerical Example}\label{sec:Numerics}

As an illustrative example, we implemented the solution of the LEQG covariance steering over a noisy double integrator, i.e., an instance of problem \eqref{LEQGcovariancesteeringproblem} with $$(A_t,B_t)\!\equiv\!\left(\!\begin{bmatrix}
0 & 1\\
0 & 0
\end{bmatrix}\!\!,\!\!\begin{bmatrix}
0\\
1\end{bmatrix}
\!\right)\!, B_t \equiv F_t, Q_t=R_t=I \: \forall t\in[0,1],$$
and randomly generated endpoint covariances
$$\Sigma_0=
    \!\begin{bmatrix}
        3.3356 & -0.0796\\
        -0.0796 & 2.0147
    \end{bmatrix}\!,\quad\Sigma_1=\!\begin{bmatrix}
        6.9230 & -2.0741\\
        -2.0741 & 3.7050
    \end{bmatrix}\!.$$
These problem data satisfy assumptions A1-A3.

With these problem data, we used Newton's method with backtracking line search to solve the implicit nonlinear equation \eqref{defG} in unknown $Y_{0\mid\theta}$. Starting with $Y_{0\mid0}^{-}$ in \eqref{Y0given0PlusMinus} as an initial guess for $Y_{0\mid\theta}$, we solved the IVP defined by \eqref{PidotSimplified} and $\Pi_{0\mid\theta}=Y_{0\mid\theta}$, followed by the IVP defined by \eqref{HdotSimplified} and $H_{0\mid\theta}=\Sigma_{0}^{-1}-\Pi_{0\mid\theta}$. Due to the one-way coupling among \eqref{HdotSimplified}-\eqref{PidotSimplified}, the latter IVP solution uses the trajectory $\Pi_{t\mid\theta}$ from the former IVP. For the Newton update, we used the Jacobian \eqref{PartialJacobianAtthetaNotEqualToZero}. Fig. \ref{Fig:covariance_funnel_3d} shows the steering of the optimally controlled state covariance $\Sigma_{t\mid\theta}$ for $\theta=-1,0,1$.
\begin{figure}
\centering
\includegraphics[width=0.88\linewidth]{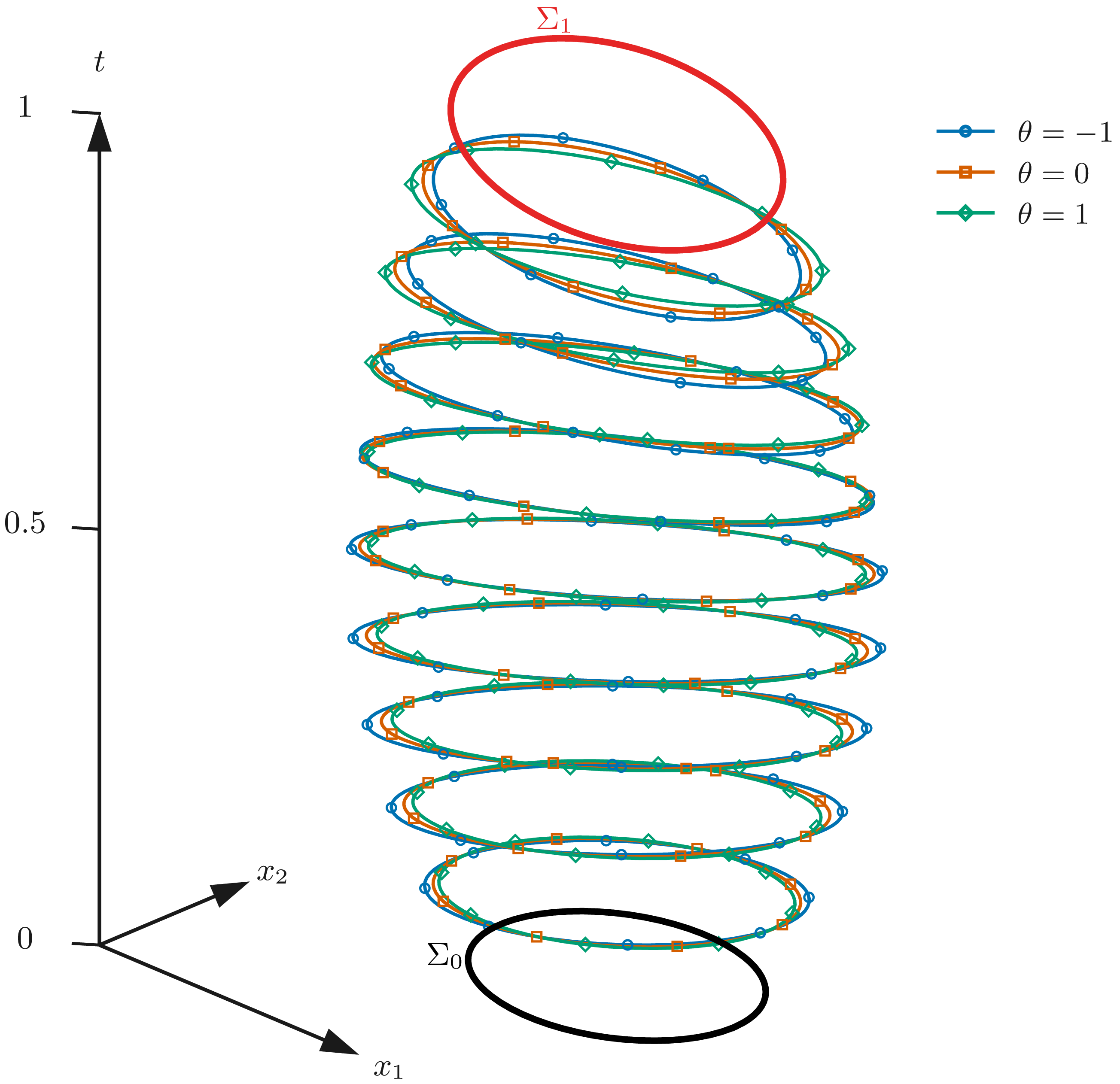}
\caption{{\small{LEQG steering of the state covariance $\Sigma_{t\mid\theta}$, shown as respective $2\sigma$ ellipses for $t\in[0,1]$, $\theta\in\{-1,0,1\}$, and state domain $[-6,6]\times[-4,4]$. All ellipses are centered at the origin.}}}
\label{Fig:covariance_funnel_3d}
\vspace*{-0.15in}
\end{figure}

As $\theta$ increases from negative to positive, the optimal solution transitions from risk-seeking to risk-averse. Consequently, $\log\det\Sigma_{t\mid\theta}$ for any fixed $t$, is a decreasing function of $\theta$. Fig. \ref{Fig:theta_versus_logdetSigma} shows this hedging against uncertainty for the optimal solution as $\theta$ increases from $-1$ to $+1$ with step-size $0.1$. 
\begin{figure}
\centering
\includegraphics[width=0.85\linewidth]{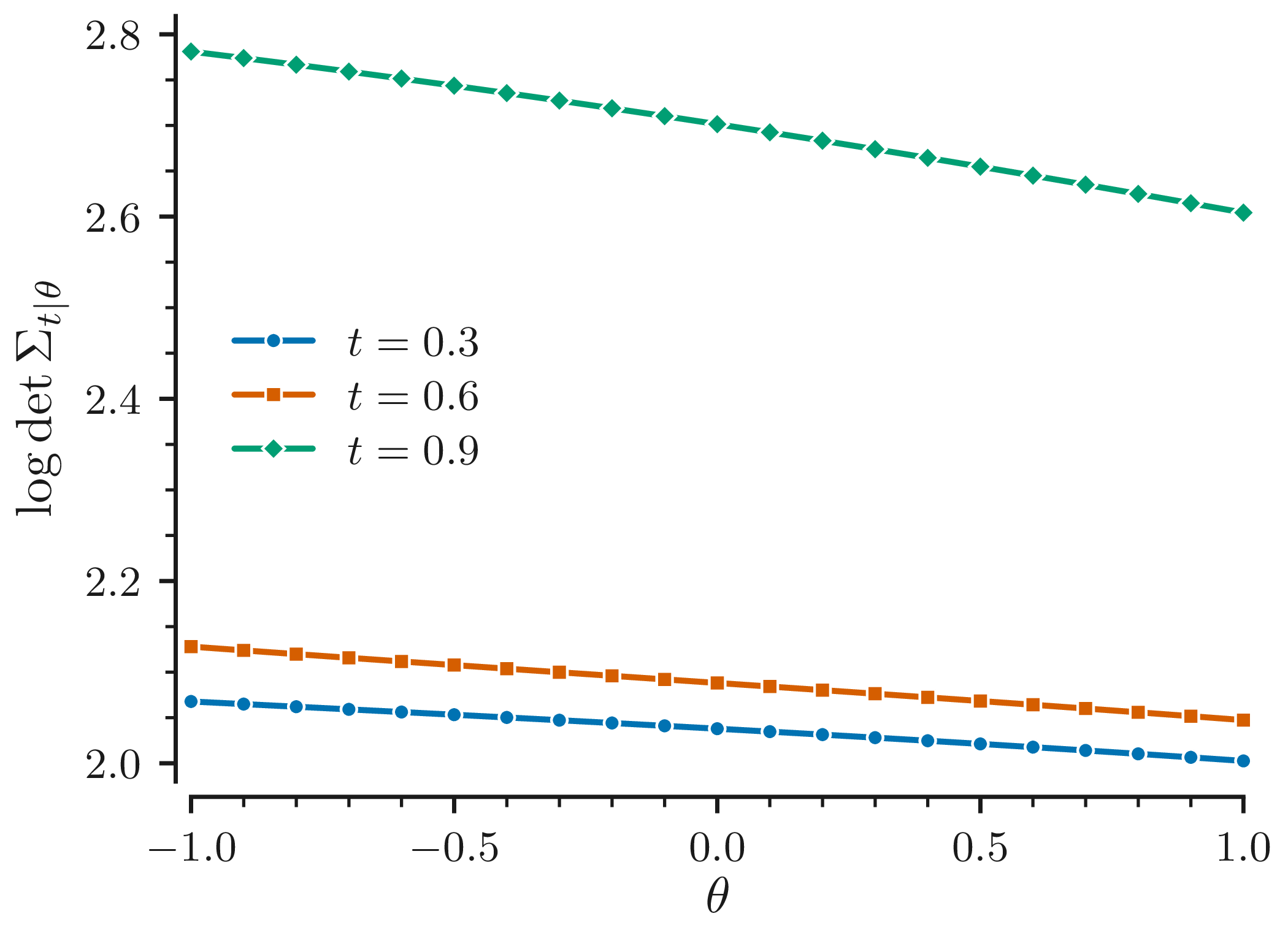}
\vspace*{-0.1in}
\caption{{\small{Risk-sensitivity parameter ($\theta$) versus the log-determinant of the optimally controlled state covariance for $t=0.3,0.6,0.9$.}}}
\label{Fig:theta_versus_logdetSigma}
\vspace*{-0.15in}
\end{figure}
Fig. \ref{Fig:theta_versus_OptimalCost} also captures the same in that the optimal cost is shown to be increasing w.r.t. $\theta$. Notice in particular that beyond the apparent scaling, the optimal $\frac{1}{\theta} \log \mathbb{E}[e^{\theta C}]$ depends on $\theta$ because the random variable $C$ in \eqref{defIQC} depends on $\Pi_{t\mid\theta}$ via \eqref{OptimalControl}, and also because the $\mathbb{E}\left[\cdot\right]$ therein is taken w.r.t. the controlled joint state distribution $\mathcal{N}\left(0,\Sigma_{t\mid\theta}\right)$.
\begin{figure}
\centering
\includegraphics[width=0.85\linewidth]{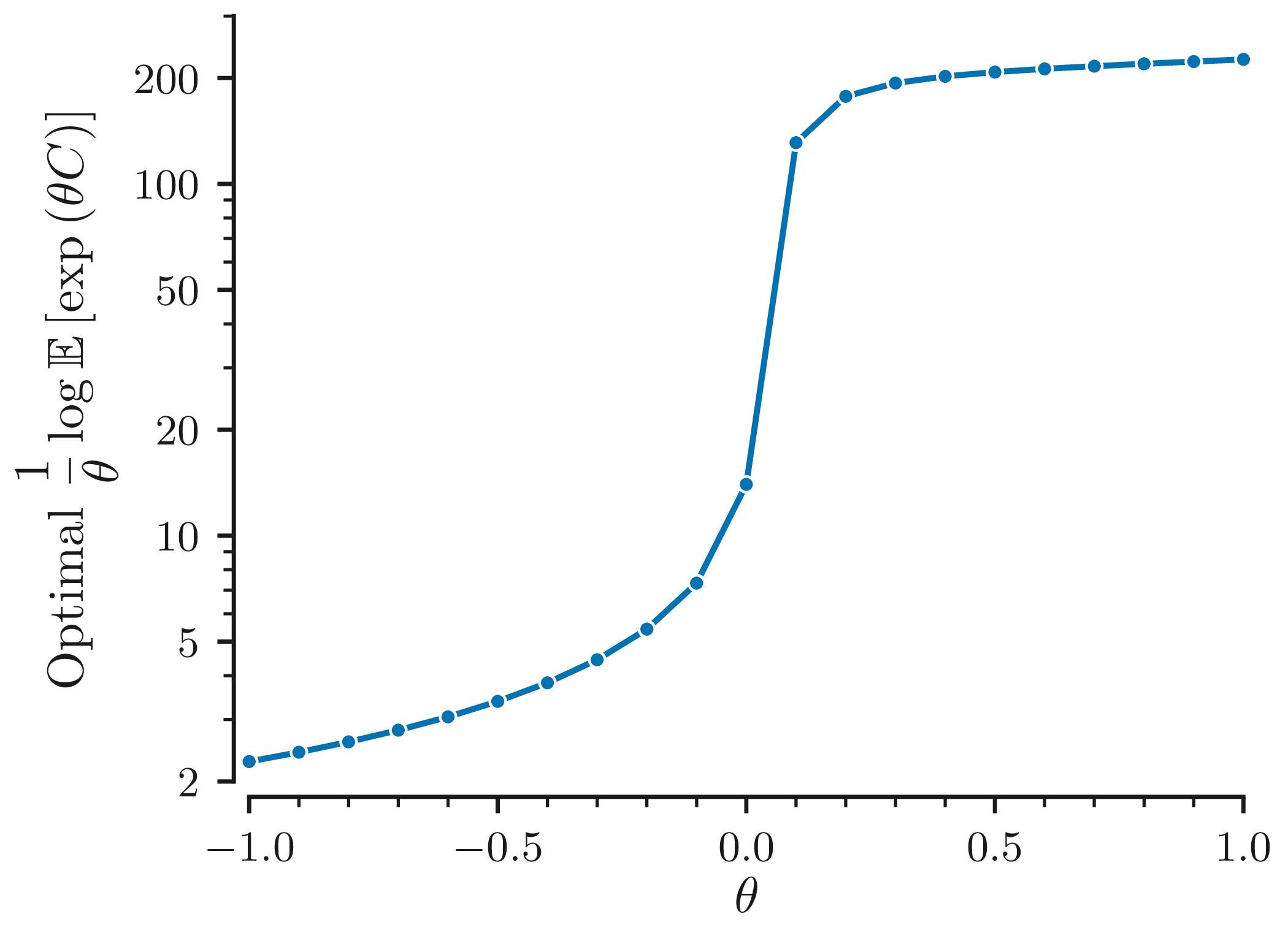}
\vspace*{-0.1in}
\caption{{\small{Risk-sensitivity parameter ($\theta$) versus the optimal cost \eqref{LEQGobjective}.}}}
\label{Fig:theta_versus_OptimalCost}
\vspace*{-0.15in}
\end{figure}

%%%%%%%%

% A fixed seed of $8642013579$ is used so that the covariance boundary conditions and multistart initializations are reproducible.
%%%%%%%%%%%%%%%%%%%%%%%%%%%%%%%%%%

\section{Concluding Remarks}\label{sec:Conclusions}
The purpose of this letter was to introduce the LEQG covariance steering problem and to explain how its solution generalizes the results known in the literature, emphasizing new structural aspects of the solution. While Theorem \ref{thm:LocalExistenceUniqueness} proved \emph{local} existence-uniqueness of solution, we suspect this result can be made \emph{global} by upgrading \eqref{PartialJacobian} with \eqref{PartialJacobianAtthetaNotEqualToZero}, and then showing that $h$ can be extended in $\mathcal{C}^{1}$ to the boundary of any finite interval. This remains ongoing work. All results and writing (including the em dashes in the abstract) are due to the three human authors, not by AI.  

%%%%%%%%%%%%%%%%%%%%%%%%%%%%%%%%%%

% \begin{figure*}
% \includegraphics[width=\linewidth]{example-image-c}
% \caption{Effect of changing $c(x,y)$.}
% \label{TBD}
% \end{figure*}

%%%%%%%%%%%%%%%%%%%%%%%%%%%%%%%%%%

% \section{Conclusions}\label{sec:Conclusions}
% {\red{TBD}}

%%%%%%%%%%%%%%%%%%%%%%%%%%%%%%%%%%%%%%%%%%%%%%%%%%%%%%%%%%%%%%%%%%%%%%%%%%%%%%%%%%%%%%%%%

\bibliographystyle{IEEEtran}
\bibliography{references.bib}

\end{document}